\documentclass{amsart}
\usepackage{mathrsfs}
\usepackage{color}
\usepackage{bbm}
\usepackage{amsmath}
\usepackage{amsfonts}
\usepackage{amssymb}
\usepackage{graphicx}

 \newtheorem{Theorem}{Theorem}[section]
 \newtheorem{Corollary}[Theorem]{Corollary}
 \newtheorem{Lemma}[Theorem]{Lemma}
 \newtheorem{Proposition}[Theorem]{Proposition}

 \newtheorem{Definition}[Theorem]{Definition}
\newtheorem{Question}[Theorem]{Question}

 \newtheorem{Remark}[Theorem]{Remark}

 \newtheorem{Example}[Theorem]{Example}
 \numberwithin{equation}{section}

\begin{document}

\title{A note on $\xi-$Bergman kernels}
\author{Shijie Bao}
\address{Shijie Bao: Institute of Mathematics, Academy of Mathematics
and Systems Science, Chinese Academy of Sciences, Beijing 100190, China.}
\email{bsjie@amss.ac.cn}

\author{Qi'an Guan}
\address{Qi'an Guan: School of
Mathematical Sciences, Peking University, Beijing 100871, China.}
\email{guanqian@math.pku.edu.cn}

\author{Zheng Yuan}
\address{Zheng Yuan: Institute of Mathematics, Academy of Mathematics
and Systems Science, Chinese Academy of Sciences, Beijing 100190, China.}
\email{yuanzheng@amss.ac.cn}

\subjclass[2010]{32D15 32E10 32L10 32U05 32W05}

\thanks{}

\keywords{$\xi-$Bergman kernel, generalized Lelong number, jumping number, complex singularity exponent}

\date{\today}

\dedicatory{}

\commby{}


\begin{abstract}
    In the present note, we introduce the $\xi-$complex singularity exponents, which come from the asymptotic property of $\xi-$Bergman kernels on sub-level sets of plurisubharmonic functions; give some relations (including a closedness property) among $\xi-$complex singularity exponents, complex singularity exponents, and jumping numbers; generalize some properties of complex singularity exponents (such as the restriction formula and subadditivity property) to $\xi-$complex singularity exponents.
    \end{abstract}
    
    \maketitle
    
    \section{Introduction}\label{Introduction}
    
    \subsection{Background}
    
      In \cite{BG1}, Bao-Guan (the first and the second named authors of the present note) introduced a version of the generalized Bergman kernels called $\xi-$Bergman kernels, and established the log-plurisubharmonicity property for fiberwise $\xi-$Bergman kernels (which is a generalization of Berndtsson's result for log-plurisubharmonicity property for fiberwise Bergman kernels, see \cite{Blogsub}), deducing an effectiveness result of the strong openness property, which is independent of the strong openness property. After that, in \cite{BG2}, Bao-Guan generalized the log-plurisubharmonicity property for fiberwise $\xi-$Bergman kernels to weighted $\xi-$Bergman kernels, deducing a new proof for the effectiveness result of the $L^p$ strong openness property in \cite{GY-lp-effe}.
    
      Recall the linear space of sequences of complex numbers
      \[\ell_1:=\left\{\xi=(\xi_{\alpha})_{\alpha\in\mathbb{N}^n} \ : \ \sum_{\alpha\in\mathbb{N}^n}|\xi_{\alpha}|\rho^{|\alpha|}<+\infty, \ \text{for \ any\ } \rho>0\right\}.\]
      
      For any $z_0\in\mathbb{C}^n$, any $\xi=(\xi_{\alpha})\in\ell_1$, and any $F(z)=\sum_{\alpha\in\mathbb{N}^n}a_{\alpha}(z-z_0)^{\alpha}\in\mathcal{O}_{n,z_0}$, the value that $\xi$ acts on $F$ is defined as 
      \[(\xi\cdot F)(z_0):=\sum_{\alpha\in\mathbb{N}^n}\xi_{\alpha}\frac{F^{(\alpha)}(z_0)}{\alpha!}=\sum_{\alpha\in\mathbb{N}^n}\xi_{\alpha}a_{\alpha}.\]
      In \cite{BG1}, it is shown that $\ell_1$ is the dual space of $\mathcal{O}_{n,z_0}$ ($\mathcal{O}_{z_0}$ for short) under the analytic Krull topology (see \cite{Dembook}).
    
       Now we recall the definition of the $\xi-$Bergman kernel. Let $D$ be a domain in $\mathbb{C}^n$, and $\psi$ a plurisubharmonic function on $D$. For any $\xi\in\ell_1$, the (weighted) $\xi-$Bergman kernel with respect to $\xi$ is denoted by
      \begin{equation*}
          K^{\psi}_{\xi,D}(z):=\sup_{f\in A^2(D,e^{-\psi})}\frac{|(\xi\cdot f)(z)|^2}{\int_D|f|^2e^{-\psi}},
      \end{equation*}
      where $A^2(D,e^{-\psi}):=\{f\in\mathcal{O}(D) : \int_D|f|^2e^{-\psi}<+\infty\}$, and we denote $K^{\psi}_{\xi,D}(z)=0$ if $A^2(D,e^{-\psi})=\{0\}$. In particular, for $\psi\equiv 0$, we denote $K_{\xi,D}^{\psi}(z)$ by $K_{\xi,D}(z)$.
    
      In \cite{BG2} (see also \cite{BG1}), the following log-plurisubharmonicity property was proved for fiberwise $\xi-$Bergman kernels, which is a generalization of Berndtsson's log-plurisubharmonicity property for fiberwise Bergman kernels (see \cite{Blogsub}). 
      
      Let $\Omega$ be a pseudoconvex domain in $\mathbb{C}^{n+1}$ with coordinate $(z,w)$, where $z\in\mathbb{C}^n$, $w\in\mathbb{C}$. Let $p$, $q$ be the natural projections $p(z,w)=w$, $q(z,w)=z$ on $\Omega$, $\Omega':=p(\Omega)$. For any $w\in \Omega'$, assume $\Omega_w:=p^{-1}(w)\subseteq\Omega$ is a bounded domain in $\mathbb{C}^n$. Let $\psi$ be a plurisubharmonic function on $\Omega$, and let $K^{\psi}_{\xi,w}(z)=K^{\psi}_{\xi,\Omega_w}(z)$ be the Bergman kernels of the domains $\Omega_w$ with respect to some fixed $\xi\in \ell_1$.
    
     \begin{Theorem}[\cite{BG1, BG2}]\label{xi-log-psh}
        $\log K^{\psi}_{\xi,w}(z)$ is a plurisubharmonic function with respect to $(z,w)$, for any $\xi\in \ell_1$.
     \end{Theorem}

     Now let $D$ be a bounded pseudoconvex domain in $\mathbb{C}^n$ such that $o\in D$, and $\varphi$ a negative plurisubharmonic function on $D$ such that $\varphi(o)=-\infty$. Let $\psi$ be a plurisubharmonic function on $D$. Denote $K_{\xi,\varphi}^{\psi}(t):=K^{\psi}_{\xi,\{\varphi<-t\}\cap D}(o)$ be the $\xi-$Bergman kernel with respect to $\xi\in\ell_1$ on the sub-level set $\{\varphi<-t\}\cap D$ at $o$ for any $t\in [0,+\infty)$. Using Theorem \ref{xi-log-psh}, it can be obtained the following proposition holds.
     
     \begin{Proposition}[\cite{BG1,BG2}]\label{log-concavity}
     $\log K_{\xi,\varphi}^{\psi}(t)$ is convex with respect to $t\in [0,+\infty)$.
     \end{Proposition}
     
    Combining with an $L^2$ estimate result of solutions of $\bar{\partial}-$equations, Bao-Guan \cite{BG2} got the following result.
     \begin{Proposition}[\cite{BG1,BG2}]\label{log-increasing}
     Assume that $\mathcal{I}(\varphi+\psi)_o\neq\mathcal{I}(\psi)_o$ and $\xi\in\ell_{\mathcal{I}(\varphi+\psi)_o}$, then $-\log K_{\xi,\varphi}^{\psi}(t)+t$ is increasing with respect to $t\in [0,+\infty)$. 
     \end{Proposition}
     
     Here for any proper ideal $I$ of $\mathcal{O}_o$, the subset $\ell_{I}$ of $\ell_1$ contains all nonzero functionals that vanish on $I$, i.e.,
     \[\ell_I:=\{\xi\in\ell_1\setminus\{0\} : (\xi\cdot F)(o)=0,\ \forall (F,o)\in I\},\]
    where $0:=\{0,\ldots,0,\ldots,\}\in\ell_1$. And we denote that $\tilde{\ell}_I:=\ell_I\cup\{0\}$ in the present note.

     \begin{Remark}
        In \cite{BG1, BG2}, combining Proposition \ref{log-increasing} with Lemma \ref{seperate} (see Section \ref{section-prep}), Bao-Guan obtained an effectiveness result of the strong openness property of multiplier ideal sheaves, which is independent of the strong openness property.
     \end{Remark}
    
     Naturally Proposition \ref{log-concavity} and \ref{log-increasing} lead us to consider the limit
     $\lim\limits_{t\rightarrow+\infty}\log K_{\xi,\varphi}^{\psi}(t)/t$
     under the assumptions of Proposition \ref{log-concavity} for functionals not only in $\ell_{\mathcal{I}(\varphi+\psi)_o}$ but also in $\ell_1\setminus \ell_{\mathcal{I}(\varphi+\psi)_o}$. Proposition \ref{log-concavity} implies that the above limit exists, and it is natural to define
     \begin{Definition}[$\xi-$complex singularity exponent with respect to $\xi$]
     \[\gamma_{o,\xi}(\varphi,\psi):=\lim_{t\rightarrow+\infty}\frac{\log K_{\xi,\varphi}^{\psi}(t)}{t}\in \{-\infty\}\cup[0,+\infty].\]
     \end{Definition}
     We may sometimes simply denote that $\gamma_{\xi}(\varphi,\psi):=\gamma_{o,\xi}(\varphi,\psi)$, and when $\psi\equiv 0$, we denote that $\gamma_{\xi}(\varphi):=\gamma_{o,\xi}(\varphi,0)$.
     
     Considering that Proposition \ref{log-concavity} and \ref{log-increasing} are useful in studying strong openness property (see \cite{BG1, BG2}), it should be convincing that there are relations between the $\xi-$complex singularity exponents  and complex singularity exponents (or log canonical thresholds). This is why we call them $\xi-$complex singularity exponents. The main results of the present note is showing some precise relations between the $\xi-$complex singularity exponents, complex singularity exponents (and jumping numbers).
     
    \subsection{Main results}
    
     Recall that $D$ is a bounded pseudoconvex domain in $\mathbb{C}^n$ such that $o\in D$, and $\varphi$ is a negative plurisubharmonic function on $D$ such that $\varphi(o)=-\infty$. Let $\psi$ be a plurisubharmonic function on $D$. Denote $K_{\xi,\varphi}^{\psi}(t):=K^{\psi}_{\xi,\{\varphi<-t\}\cap D}(o)$ be the $\xi-$Bergman kernel for any $t\in [0,+\infty)$ with respect to any functional $\xi\in\ell_1$.
     
     Actually, the $\xi-$complex singularity exponent (\emph{$\xi-$cse} for short) $\gamma_{\xi}(\varphi,\psi)$ is independent of the choice of the bounded pseudoconvex domain $D$, which is why we do not include $D$ in the notation $\gamma_{\xi}(\varphi,\psi)$. More precisely, we have
     \begin{Theorem}\label{Thm-xicse-local}
    If $\xi\in\tilde{\ell}_{\mathcal{I}(\psi)_o}$ then $\gamma_{\xi}(\varphi,\psi)=-\infty$, and if $\xi\in\ell_1\setminus\tilde{\ell}_{\mathcal{I}(\psi)_o}$, then
        \[\gamma_{\xi}(\varphi,\psi)=\inf\{c\geq 0 : \xi\in \ell_{\mathcal{I}(c\varphi+\psi)_o}\}.\]
     \end{Theorem}
     
    Theorem  \ref{Thm-xicse-local} shows that the asymptotic property of the $\xi-$Bergman kernels on sub-level sets of plurisubharmonic functions actually comes from how the functionals behave.
    
    Next we give some results which show the relations between the $\xi-$cse and the complex singularity exponents and jumping numbers.
    
    Let $(F,o)$ be a holomorphic germ near $o$, and denote that
    \[c_o^F(\varphi,\psi):=\sup\{c\geq 0 : |F|^2e^{-c\varphi-\psi} \ \text{is\ locally\ } L^1 \ \text{near\ } o\}.\]
   We give the following result.
   
    \begin{Theorem}\label{thm-gln-jn}
   Assume $c_o^F(\varphi,\psi)>0$, then
    \[c^F_o(\varphi,\psi)=\inf_{\xi\in\ell_1, \ (\xi\cdot F)(o)\neq 0}\gamma_{\xi}(\varphi,\psi).\]
    \end{Theorem}

     Particularly, denote the complex singularity exponent (\emph{cse} for short, see \cite{Tian})
     \[c_o(\varphi):=\sup\{c\geq 0 : e^{-c\varphi} \ \text{is\ locally\ } L^1 \ \text{near\ } o\}.\]
     We can obtain the following corollary of Theorem \ref{thm-gln-jn} related to the $\xi-$cse and cse.
    
    \begin{Corollary}\label{cor-gln-lct}
    Assume $c_o(\varphi)>0$, then
    \[c_o(\varphi)=\inf_{\xi\in \ell_1\setminus\{0\}}\gamma_{\xi}(\varphi).\]
    \end{Corollary}
    
    According to Theorem \ref{thm-gln-jn} and Corollary \ref{cor-gln-lct}, it is natural to ask:
    
    \begin{Question}\label{Q:inf=min}
    Is the \emph{$\inf$} in Theorem \ref{thm-gln-jn} (or Corollary \ref{cor-gln-lct}) actually \emph{$\min$}, i.e., can it be achieved by some functional $\xi$?
    \end{Question}
    
    We answer the Question \ref{Q:inf=min} affirmatively. More precisely, we prove the following result related to Theorem \ref{thm-gln-jn}.

    \begin{Theorem}\label{thm-closed-jn}
        Assume $F$ is a local holomorphic germ near $o$ such that $c_o^F(\varphi,\psi)>0$. Then there exists $\xi\in\ell_1$ such that $(\xi\cdot F)(o)\neq 0$ and
         \[c_o^F(\varphi,\psi)=\gamma_{\xi}(\varphi,\psi).\]
        \end{Theorem}
    
    Related to Corollary \ref{cor-gln-lct}, we obtain the following corollary of Theorem \ref{thm-closed-jn}.

    \begin{Corollary}\label{cor-closed-lct}
        Assume $c_o(\varphi)>0$, then there exists $\xi\in\ell_1$ such that
     \[c_o(\varphi)=\gamma_{\xi}(\varphi).\]
    \end{Corollary}

    The functional $\xi\in\ell_1$ in Corollary \ref{cor-closed-lct} can actually be chosen as $\xi=(1,0,\ldots,0,\ldots)$, which could imply an asymptotic property for the original Bergman kernels on sub-level sets of plurisubharmonic functions (see Section \ref{ClosedSection}, Example \ref{ex-xi0} and Corollary \ref{Cor-limlogK(t)/t}).

    We would like to call that the affirmative answer of Question \ref{Q:inf=min} is the closedness property of $\xi-$cse. We would also state the relations between Question \ref{Q:inf=min} and the strong openness property of multiplier ideal sheaves in this note.

    \emph{The above results show that the local properties of plurisubharmonic functions at singular points can be described by the asymptotics of the global objects: such as the $\xi-$Bergman kernels.} 
    
    In the present note, we shall also give some calculations and examples for $\xi-$Bergman kernels and $\xi-$cse. We will do some calculations on the $\xi-$Bergman kernels and $\xi-$cse related to toric plurisubharmonic functions in Section \ref{section_toric}. The valuative point of view on the $\xi-$cse is shown in Section \ref{valuative-section}. Some classical results will be reproved in these two sections as examples (see Corollary \ref{int(P(varphi))} in Section \ref{section_toric}, and Corollary \ref{cor-nu} in Section \ref{valuative-section}).

    In addition, we will also generalize some properties of cse to $\xi-$cse (see Section \ref{restriction-section}). For example, we give the restriction formula (Theorem \ref{thm-restriction}) and the subadditivity property (Theorem \ref{thm-subadd}) of the $\xi-$cse which is similar to those of the complex singularity exponents.
    
    \section{Preparations}\label{section-prep}
    We do some preparations in this section for proving the main results.
    
    We need the following Lemma.
    \begin{Lemma}\label{Lem-jet_ext}
        Let $D$ be a bounded pseudoconvex domain in $\mathbb{C}^n$ such that $o\in D$, and $\phi$ a plurisubharmonic function on $D$. Let $\xi\in\ell_1$. If $\xi\in\ell_1\setminus\tilde{\ell}_{\mathcal{I}(\phi)_o}$, then there exists $F\in A^2(D,e^{-\phi})$ such that $(\xi\cdot F)(o)\neq 0$. 
    \end{Lemma}
    
    To prove Lemma \ref{Lem-jet_ext}, we recall the following lemma, which was used in \cite{BG1, BG2} to prove Proposition \ref{log-increasing}.
    
    \begin{Lemma}[see \cite{GZeffectiveness, G16}]\label{L2method}
        Let $D$ be a pseudoconvex domain in $\mathbb{C}^n$ such that $o\in D$. Let $\varphi$ be a negative plurisubharmonic function on $D$ such that $\varphi(o)=-\infty$. Let $\psi$ be a plurisubharmonic function on $D$. Let $t_0\in (0,+\infty)$, and $F$ a holomorphic function on $\{\varphi<-t_0\}$ such that
        \begin{equation*}
            \int_{\{\varphi<-t_0\}}|F|^2e^{-\psi}<+\infty.
        \end{equation*}
    Then there exists a holomorphic function $\tilde{F}$ on $D$ such that
    \[(\tilde{F}-F,o)\in\mathcal{I}(\varphi+\psi)_o,\]
    and
    \begin{equation}
        \int_D|\tilde{F}-(1-b_{t_0}(\varphi))F|^2e^{-\psi}\leq C\int_D\mathbb{I}_{\{-t_0-1<\varphi<-t_0\}}|F|^2e^{-\varphi-\psi},
    \end{equation}
    where $b_{t_0}(t)=\int_{-\infty}^t\mathbb{I}_{\{-t_0-1<s<-t_0\}}\mathrm{d}s$, and $C$ is a positive constant.
    \end{Lemma}
    
    Now we prove Lemma \ref{Lem-jet_ext}.
    \begin{proof}[Proof of Lemma \ref{Lem-jet_ext}:]
    Since $\xi\in\ell_1\setminus\tilde{\ell}_{\mathcal{I}(\phi)_o}$, there exists some $(f,o)\in\mathcal{I}(\phi)_o$ such that $(\xi\cdot f)(o)\neq 0$. Then we can find some $r>0$ such that $\Delta^n_r\subset\subset D$ and
    \[\int_{\Delta^n_r}|f|^2e^{-\phi}<+\infty.\]
    We can also find some $R>0$ such that $\Delta^n_R\supset D$ since $D$ is bounded. Let $N\in\mathbb{N}_+$ sufficiently large, and denote
    \[\Psi:=2(N+n)\log\left(\frac{|z|}{R}\right),\]
    which is a negative plurisubharmonic function on $D$ such that $\Psi(o)=-\infty$. Let $t_0:=2(N+n)\log\left(\frac{R}{r}\right)\in (0,+\infty)$. Then
    \[\int_{\{\Psi<-t_0\}}|f|^2e^{-\phi}=\int_{\Delta^n_r}|f|^2e^{-\phi}<+\infty.\]
    
    Now it follows from Lemma \ref{L2method} that there exists a holomorphic function $F_N$ on $D$ such that $(F_N-f,o)\in\mathcal{I}(\Psi+\phi)_o$, and
    \begin{equation}\label{proof_jet_equa1}
        \int_D|F_N-(1-b_{t_0}(\Psi))f|^2e^{-\phi}\leq C\int_D\mathbb{I}_{\{-t_0-1<\Psi<-t_0\}}|f|^2e^{-\Psi-\phi}.
    \end{equation}
    On the one hand, we have
    \begin{equation}\label{proof_jet_equa2}
        \begin{split}
        \left(\int_D|F_N-(1-b_{t_0}(\Psi))f|^2e^{-\phi}\right)^{1/2}&\geq \left(\int_D|F_N|^2e^{-\phi}\right)^{1/2}-\left(\int_D|(1-b_{t_0}(\Psi))f|^2e^{-\phi}\right)^{1/2}\\
        &\geq \left(\int_D|F_N|^2e^{-\phi}\right)^{1/2}-\left(\int_{\{\Psi<-t_0\}}|f|^2e^{-\phi}\right)^{1/2}.
        \end{split}
    \end{equation}
    On the other hand, it is clear that
    \begin{equation}\label{proof_jet_equa3}
        \begin{split}
        \int_D\mathbb{I}_{\{-t_0-1<\Psi<-t_0\}}|f|^2e^{-\Psi-\phi}&\leq e^{t_0+1}\int_{\{\Psi<-t_0\}}|f|^2e^{-\phi}\\
        &=e\left(\frac{R}{r}\right)^{2N+2n}\int_{\{\Psi<-t_0\}}|f|^2e^{-\phi}.
        \end{split}
    \end{equation}
    Then according to (\ref{proof_jet_equa1}), (\ref{proof_jet_equa2}), and (\ref{proof_jet_equa3}), we have
    \begin{equation}\label{proof_jet_equa4}
        \int_D|F_N|^2e^{-\phi}\leq C_1\left(\frac{R}{r}\right)^{2N+2n}\int_{\Delta^n_r}|f|^2e^{-\phi}<+\infty,
    \end{equation}
    where $C_1$ is a positive constant independent of $N$. 
    
    Since $(F_N-f,o)\in\mathcal{I}(\Psi+\phi)_o\subset\mathcal{I}(\Psi)_o$ and $\Psi=2(N+n)\log\left(\frac{|z|}{R}\right)$, we can know that
    \[(F_N-f)^{(\alpha)}(o)=0, \ \forall |\alpha|\leq N.\]
    In addition, we write
    \[f(z)=\sum_{\alpha\in\mathbb{N}^n}a_{\alpha}z^{\alpha},\]
    and
    \[F_N(z)=\sum_{\alpha\in\mathbb{N}^n}c_{\alpha,N}z^{\alpha}\]
    on $\Delta^n_r$. Then $a_{\alpha}=c_{\alpha,N}$ for any $|\alpha|\leq N$. According to Cauchy's inequality, we have the estimate
    \[|c_{\alpha,N}|\leq \frac{M_N}{(r/3)^{|\alpha|}},\]
    where
    \[M_N:=\sup_{z\in\overline{\Delta^n_{r/3}}}|F_N(z)|\leq \frac{C_2}{(\pi r^2/9)^n}\int_{\Delta}|F_N|^2e^{-\phi}\leq \frac{C_1C_2}{(\pi r^2/9)^n}\left(\frac{R}{r}\right)^{2N+2n}\int_{\Delta^n_r}|f|^2e^{-\phi},\]
    and $C_2:=\sup_{\Delta^n_{2r/3}}e^{\phi}<+\infty$.
    Thus,
    \[|c_{\alpha,N}|\leq \frac{C_3}{r^{2n+|\alpha|}}\left(\frac{R}{r}\right)^{2N+2n}\int_{\Delta^n_r}|f|^2e^{-\phi}\]
    for some positive constant $C_3$ independent of $N$. Set $\xi=(\xi_{\alpha})_{\alpha}$, then we can obtain
    \begin{equation*}
        \begin{split}
            \sum_{|\alpha|>N}|\xi_{\alpha}c_{\alpha,N}|&\leq \frac{C_3 R^{2n}}{r^{4n}}\sum_{|\alpha|>N}\frac{\xi_{\alpha}}{r^{|\alpha|}}\left(\frac{R}{r}\right)^{2N}\int_{\Delta^n_r}|f|^2e^{-\phi}\\
            &\leq \frac{C_3 R^{2n}}{r^{4n}}\sum_{|\alpha|>N}\xi_{\alpha}\left(\frac{R^2}{r^3}\right)^{|\alpha|}\int_{\Delta^n_r}|f|^2e^{-\phi}.
        \end{split}
    \end{equation*}
    Since $\xi\in\ell_1$, we have
    \[\sum_{\alpha\in\mathbb{N}^n}\xi_{\alpha}\left(\frac{R^2}{r^3}\right)^{|\alpha|}<+\infty,\]
    which implies that there exists some $N_1\in\mathbb{N}_+$, such that
    \[\sum_{|\alpha|>N}|\xi_{\alpha}c_{\alpha,N}|<\frac{1}{3}|(\xi\cdot f)(o)|\]
    for any $N>N_1$, where $|(\xi\cdot f)(o)|>0$ by the assumption. Additionally, there exists some $N_2\in\mathbb{N}_+$, such that for any $N>N_2$,
    \[\sum_{|\alpha|>N}|\xi_{\alpha}a_{\alpha}|<\frac{1}{3}|(\xi\cdot f)(o)|.\]
    Then for any $N>\max\{N_1,N_2\}$, it follows from $a_{\alpha}=c_{\alpha,N}$ for any $|\alpha|\leq N$ that
    \begin{equation*}
        \begin{split}
            |(\xi\cdot F_N)(o)|&=\left|\sum_{\alpha
            \in\mathbb{N}^n}\xi_{\alpha}c_{\alpha,N}\right|\\
            &\geq \left|\sum_{|\alpha|\leq N}\xi_{\alpha}c_{\alpha,N}\right|-\sum_{|\alpha|>N}|\xi_{\alpha}c_{\alpha,N}|\\
            &=\left|\sum_{|\alpha|\leq N}\xi_{\alpha}a_{\alpha}\right|-\sum_{|\alpha|>N}|\xi_{\alpha}c_{\alpha,N}|\\
            &\geq |(\xi\cdot f)(o)|-\sum_{|\alpha|>N}|\xi_{\alpha}c_{\alpha,N}|-\sum_{|\alpha|>N}|\xi_{\alpha}a_{\alpha}|\\
            &>|(\xi\cdot f)(o)|-\frac{1}{3}|(\xi\cdot f)(o)|-\frac{1}{3}|(\xi\cdot f)(o)|\\
            &>0.
        \end{split}
    \end{equation*}
    We have $F_N\in A^2(D,e^{-\phi})$ according to (\ref{proof_jet_equa4}). Thus $F=F_N$ is just what we need.
    \end{proof}
    
    We state some simple properties of the $\xi-$cse.
    \begin{Lemma}\label{ctimes}
    For any $a\in\mathbb{R}$ and $c\in\mathbb{R}_+$,
    
    (i) $\gamma_{\xi}(\varphi+a,\psi)=\gamma_{\xi}(\varphi,\psi)$;
    
    (ii) $\gamma_{\xi}(c\varphi,\psi)=\frac{1}{c}\gamma_{\xi}(\varphi,\psi)$.
    \end{Lemma}
    
    \begin{proof}
    (i) Note that for any $a\in\mathbb{R}$, $t\geq 0$ and $\xi\in\ell_1$,
    \[K^{\psi}_{\xi,\varphi+a}(t)=K^{\psi}_{\xi,\varphi}(t+a),\]
    then
    \[\gamma_{\xi}(\varphi+a,\psi)=\lim_{t\rightarrow+\infty}\frac{\log K^{\psi}_{\xi,\varphi}(t+a)}{t}=\lim_{t\rightarrow+\infty}\frac{\log K^{\psi}_{\xi,\varphi}(t+a)}{t+a}=\gamma_{\xi}(\varphi,\psi).\]
    
    (ii) Note that for any $c>0$, $t\geq 0$ and $\xi\in\ell_1$,
    \[K^{\psi}_{\xi,c\varphi}(t)=K^{\psi}_{\xi,\varphi}(t/c),\]
    then
    \[\gamma_{\xi}(c\varphi,\psi)=\lim_{t\rightarrow+\infty}\frac{\log K^{\psi}_{\xi,c\varphi}(t)}{t}=\frac{1}{c}\lim_{t\rightarrow+\infty}\frac{\log K^{\psi}_{\xi,\varphi}(t)}{t}=\frac{1}{c}\gamma_{\xi}(\varphi,\psi).\]
    \end{proof}

    Denote that $\ell_{-\infty}(\varphi,\psi):=\{\xi\in\ell_1 : \gamma_{\xi}(\varphi,\psi)=-\infty\}$. Using  Proposition \ref{log-concavity}, Proposition \ref{log-increasing}, and Lemma \ref{Lem-jet_ext}, we obtain the following lemma.
    
    \begin{Lemma}\label{gamma>1}
    We have
    
    (i) $\ell_{-\infty}(\varphi,\psi)=\tilde{\ell}_{\mathcal{I}(\psi)_o}$;
    
    (ii) if $\xi\in\ell_{\mathcal{I}(\varphi+\psi)_o}\setminus\ell_{-\infty}(\varphi,\psi)$, then $\gamma_{\xi}(\varphi,\psi)\in [0,1]$;
    
    (iii) if $\xi\in \ell_1\setminus\tilde{\ell}_{\mathcal{I}(\varphi+\psi)_o}$, then $\gamma_{\xi}(\varphi,\psi)\in [1,+\infty]$.
    \end{Lemma}
    
    \begin{proof}
    Firstly we prove statement (i). Since for any $\xi\in\tilde{\ell}_{\mathcal{I}(\psi)_o}$, and $F\in A^2(\{\varphi<-t\}\cap D,e^{-\psi})$, we have $(\xi\cdot F)(o)=0$, it is clear that $\ell_{-\infty}(\varphi,\psi)\supset\tilde{\ell}_{\mathcal{I}(\psi)_o}$.
    
    For the other side, if $\ell_{-\infty}(\varphi,\psi)\subset\tilde{\ell}_{\mathcal{I}(\psi)_o}$ is not true, then there exists $\xi\in \ell_{-\infty}(\varphi,\psi)$ such that $\xi\notin \tilde{\ell}_{\mathcal{I}(\psi)_o}$. By Lemma \ref{Lem-jet_ext}, there exists $F\in A^2(D,e^{-\psi})$ such that $(\xi\cdot F)(o)\neq 0$. Now we have
    \[K^{\psi}_{\xi,\varphi}(t)=K^{\psi}_{\xi,\{\varphi<-t\}\cap D}(o)\geq \frac{|(\xi\cdot F)(o)|^2}{\int_{\{\varphi<-t\}\cap D}|F|^2e^{-\psi}}>0\]
    for any $t\geq 0$, which implies
    \[\log K^{\psi}_{\xi,\varphi}(t)>-\infty, \ \forall t\geq 0.\]
    Then $\gamma_{\xi}(\varphi,\psi)>-\infty$, which is a contradiction to $\xi\in \ell_{-\infty}(\varphi,\psi)$. Thus $\ell_{-\infty}(\varphi,\psi)=\tilde{\ell}_{\mathcal{I}(\psi)_o}$.
    
    Secondly, we prove statement (ii). In fact, according to Proposition \ref{log-concavity} and Proposition \ref{log-increasing}, we know that $\log K^{\psi}_{\xi,\varphi}(t)$ is convex and $-\log K^{\psi}_{\xi,\varphi}(t)+t$ is increasing for any $\xi\in \ell_{\mathcal{I}(\varphi+\psi)_o}$. It means that $\gamma_{\xi}(\varphi,\psi)\leq 1$. Then $\gamma_{\xi}(\varphi,\psi)\in [0,1]$ for $\xi\in\ell_{\mathcal{I}(\varphi+\psi)_o}\setminus\ell_{-\infty}(\varphi,\psi)$.
    
    Finally, we prove statement (iii). Statement (i) shows $\gamma_{\xi}(\varphi,\psi)>-\infty$. Since $\xi\in \ell_1\setminus\tilde{\ell}_{\mathcal{I}(\varphi+\psi)_o}$, according to Lemma \ref{Lem-jet_ext}, we can find some $F\in A^2(D,e^{-\varphi-\psi})$ such that $(\xi\cdot F)\neq 0$. Then for any $t\geq 0$, we have
    \[C\geq\int_{\{\varphi<-t\}\cap D}|F|^2e^{-\varphi-\psi}\geq e^t\int_{\{\varphi<-t\}\cap D}|F|^2e^{-\psi}\]
    for some positive constant $C$, which implies
    \[K^{\psi}_{\xi,\varphi}(t)=K^{\psi}_{\xi,\{\varphi<-t\}\cap D}(o)\geq \frac{|(\xi\cdot F)(o)|^2}{\int_{\{\varphi<-t\}\cap D}|F|^2e^{-\psi}}\geq C^{-1}|(\xi\cdot F)(o)|^2e^t.\]
    Thus we get
    \[\frac{\log K^{\psi}_{\xi,\varphi}(t)}{t}\geq 1+\frac{\log(C^{-1}|(\xi\cdot F)(o)|^2)}{t}\rightarrow 1, \ t\rightarrow +\infty.\]
    It follows that $\gamma_{\xi}(\varphi,\psi)\in [1,+\infty]$.
    \end{proof}
    
    We also need to recall a lemma used in \cite{BG1, BG2}.
    \begin{Lemma}[see \cite{BG1}]\label{seperate}
    Let $I$ be an ideal of $\mathcal{O}_o$ such that $I\neq\mathcal{O}_o$. Let $(F,o)\in\mathcal{O}_o$ such that $(F,o)\notin I$. Then there exists $\xi\in \ell_{I}$ such that $(\xi\cdot F)(o)\neq 0$.
    \end{Lemma}

    \section{Proofs of Theorem \ref{Thm-xicse-local} and Theorem \ref{thm-gln-jn}}\label{Proofmainthm}
    
    We first give the proof of Theorem \ref{Thm-xicse-local}.
    \begin{proof}[Proof of Theorem \ref{Thm-xicse-local}]
    The first part of Theorem \ref{Thm-xicse-local} comes from Lemma \ref{gamma>1}(i). In the following, we assume $\xi\in\ell_1\setminus\tilde{\ell}_{\mathcal{I}(\psi)_o}$, which implies $\gamma_{\xi}(\varphi,\psi)>-\infty$.
    
    Let $c\in [0,+\infty)$. Suppose $\xi\in\ell_{\mathcal{I}(c\varphi+\psi)_o}$. According to Lemma \ref{gamma>1}(ii), we have $\gamma_{\xi}(c\varphi,\psi)\in [0,1]$. Then Lemma \ref{ctimes}(ii) implies $\gamma_{\xi}(\varphi,\psi)\leq c$. We get that
    \begin{equation}\label{gamma>inf}
        \gamma_{\xi}(\varphi,\psi)\leq \inf\{c\geq 0 : \xi\in \ell_{\mathcal{I}(c\varphi+\psi)_o}\}.
    \end{equation}
    
    Suppose $c'\in [0,+\infty)$ satisfies $\xi\in\ell_1\setminus\tilde{\ell}_{\mathcal{I}(c'\varphi+\psi)_o}$. Then Lemma \ref{gamma>1}(iii) implies $\gamma_{\xi}(c'\varphi,\psi)\geq 1$. Similarly, it follows that $\gamma_{\xi}(\varphi,\psi)\geq c'$ by Lemma \ref{ctimes}(ii). Combining with (\ref{gamma>inf}), we finally get
    \begin{equation*}
        \gamma_{\xi}(\varphi,\psi)= \inf\{c\geq 0 : \xi\in \ell_{\mathcal{I}(c\varphi+\psi)_o}\}.
    \end{equation*}
    \end{proof}
    
    Next we give the proof of Theorem \ref{thm-gln-jn}.
    \begin{proof}[Proof of Theorem \ref{thm-gln-jn}]
    According to $c_o^F(\varphi,\psi)>0$, we know $\xi\in\ell_1\setminus\tilde{\ell}_{\mathcal{I}(\psi)_o}$ for any $\xi\in\ell_1$ satisfying $(\xi\cdot F)(o)\neq 0$, and it follows from Lemma \ref{gamma>1}(i) that $\gamma_{\xi}(\varphi,\psi)\in [0,+\infty]$ for any $\xi\in\ell_1\setminus\tilde{\ell}_{\mathcal{I}(\psi)_o}$.
    
    For any $c<c_o^F(\varphi,\psi)$, we know $(F,o)\in \mathcal{I}(c\varphi+\psi)_o$. Then for any $\xi\in\ell_1\setminus\tilde{\ell}_{\mathcal{I}(\psi)_o}$ with $(\xi\cdot F)(o)\neq 0$ (of course there exists such $\xi$ according to Lemma \ref{seperate}), we have $\xi\in\ell_1\setminus\tilde{\ell}_{\mathcal{I}(c\varphi+\psi)_o}$, thus Theorem \ref{Thm-xicse-local} implies $\gamma_{\xi}(\varphi,\psi)\geq c$. It follows that
    \[c_o^F(\varphi,\psi)\leq\inf_{\xi\in\ell_1, \ (\xi\cdot F)(o)\neq 0}\gamma_{\xi}(\varphi,\psi).\]
    
    For any $c>c_o^F(\varphi,\psi)$, we show that there exists some $\xi\in\ell_1$ with $(\xi\cdot F)(o)\neq 0$ such that $\gamma_{\xi}(\varphi,\psi)\leq c$. Note that $c>c_o^F(\varphi,\psi)$ implies that $(F,o)\notin\mathcal{I}(c\varphi+\psi)_o\subsetneq \mathcal{O}_o$. Then Lemma \ref{seperate} gives us some functional $\xi\in \ell_{\mathcal{I}(c\varphi+\psi)_o}$ with $(\xi\cdot F)(o)\neq 0$. Then according to Theorem \ref{Thm-xicse-local}, we can get $\gamma_{\xi}(\varphi,\psi)\leq c$.
    
    Now the proof of
    \[c_o^F(\varphi,\psi)=\inf_{\xi\in\ell_1, \ (\xi\cdot F)(o)\neq 0}\gamma_{\xi}(\varphi,\psi)\]
    is done.
    \end{proof}

    Denote
    \[c_o(\varphi,\psi):=\sup\{c\geq 0 : e^{-c\varphi-\psi} \ \text{is\ locally\ } L^1 \ \text{near\ } o\},\]
    and the complex singularity exponent (\emph{cse} for short, see \cite{Tian})
    \[c_o(\varphi):=\sup\{c\geq 0 : e^{-c\varphi} \ \text{is\ locally\ } L^1 \ \text{near\ } o\}.\]
   
    We have the following property related to the $\xi-$cse and $c_o(\varphi,\psi)$ as a corollary of Theorem \ref{thm-gln-jn}.
    
    \begin{Corollary}\label{thm-gln-lct}
    Assume $c_o(\varphi,\psi)>0$, then
    \[c_o(\varphi,\psi)=\inf_{\xi\in\ell_1\setminus\{0\}}\gamma_{\xi}(\varphi,\psi).\]
    \end{Corollary}
   
    \begin{proof}
    Note that
    \[c_o(\varphi,\psi)=c_o^{1}(\varphi,\psi)=\inf_{(F,o)\in\mathcal{O}_o\setminus\{0\}}c_o^F(\varphi,\psi),\]
    and for any $\xi\in\ell_1\setminus\{0\}$, there must exist some $(F,o)\in\mathcal{O}_o$ satisfying $(\xi\cdot F)(o)\neq 0$. Then it follows from Theorem \ref{thm-gln-jn} that
    \begin{equation*}
        \begin{split}
            c_o(\varphi,\psi)&=\inf_{(F,o)\in\mathcal{O}_o\setminus\{0\}}c_o^F(\varphi,\psi)\\
            &=\inf_{(F,o)\in\mathcal{O}_o\setminus\{0\}}\left(\inf_{\xi\in\ell_1, \ (\xi\cdot F)(o)\neq 0}\gamma_{\xi}(\varphi,\psi)\right)\\
            &=\inf_{\xi\in\ell_1\setminus\{0\}}\gamma_{\xi}(\varphi,\psi).
        \end{split}
    \end{equation*}

    \end{proof}

   When $\psi\equiv 0$, Corollary \ref{thm-gln-lct} degenerates to Corollary \ref{cor-gln-lct}, the complex singularity exponents (or log canonical thresholds) case.
   
   Moreover, let $(F,o)$ be a holomorphic germ near $o$, and denote the jumping number
    \[c_o^F(\varphi):=\sup\{c\geq 0 : |F|^2e^{-c\varphi} \ \text{is\ locally\ } L^1 \ \text{near\ } o\}.\]
    Theorem \ref{thm-gln-jn} degenerates to the jumping number case when $\psi\equiv 0$.
   
   \begin{Corollary}\label{cor-gln-jn}
   Assume $c_o^F(\varphi)>0$, then
   \[c_o^F(\varphi)=\inf_{\xi\in \ell_1, \ (\xi\cdot F)(o)\neq 0}\gamma_{\xi}(\varphi).\]
   \end{Corollary}

    \section{Closedness properties for $\xi-$cse}\label{ClosedSection}
    We answer the Question \ref{Q:inf=min} affirmatively in this section, i.e., the proofs of Theorem \ref{thm-closed-jn} and Corollary \ref{cor-closed-lct}. For some functional $\xi$ achieving $c_o^F(\varphi,\psi)$ (or $c_o(\varphi,\psi)$, $c_o^F(\varphi)$, and $c_o(\varphi)$), we would like to call that $\xi$ computes $c_o^F(\varphi,\psi)$ (or $c_o(\varphi,\psi)$, $c_o^F(\varphi)$, and $c_o(\varphi)$ respectively). 

    We prove Theorem \ref{thm-closed-jn}.

    \begin{proof}
        Denote that
        \[\mathcal{I}_+(c\varphi+\psi)_o:=\bigcup_{c'>c}\mathcal{I}(c'\varphi+\psi)_o\]
        for any $c\in [0,+\infty)$. Then by the definition of $c_o^F(\varphi,\psi)$, it is clear that
        \[(F,o)\notin\mathcal{I}_+(c_o^F(\varphi,\psi)\varphi+\psi)_o.\]
        Now we can choose some $\xi\in\ell_{\mathcal{I}_+(c_o^F(\varphi,\psi)\varphi+\psi)_o}\setminus\tilde{\ell}_{\mathcal{I}(\psi)_o}$ such that $(\xi\cdot F)(o)\neq 0$ according to Lemma \ref{seperate}. Note that $\xi\in\ell_{\mathcal{I}(c\varphi+\psi)_o}\setminus\tilde{\ell}_{\mathcal{I}(\psi)_o}$ for any $c>c_o^F(\varphi,\psi)$. It follows from Lemma \ref{gamma>1}(ii) and Lemma \ref{ctimes}(ii) that $\gamma_{\xi}(\varphi,\psi)\leq c$ for any $c>c_o^F(\varphi,\psi)$, which implies $\gamma_{\xi}(\varphi,\psi)\leq c_o^F(\varphi,\psi)$. In addition, Theorem \ref{thm-gln-jn} shows that $\gamma_{\xi}(\varphi,\psi)\geq c_o^F(\varphi,\psi)$. Then we have $\gamma_{\xi}(\varphi,\psi)=c_o^F(\varphi,\psi)$, and such $\xi$ is just the functional that computes $c_o^F(\varphi,\psi)$.
        
        \end{proof}

    Related to Corollary \ref{thm-gln-lct}, we can obtain the following corollary by letting $F\equiv 1$ in Theorem \ref{thm-closed-jn}.
    
    \begin{Corollary}\label{thm-closed-lct}
        Assume $c_o(\varphi,\psi)>0$, then there exists $\xi\in\ell_1$ such that
         \[c_o(\varphi,\psi)=\gamma_{\xi}(\varphi,\psi).\]
        \end{Corollary}

    When $\psi\equiv 0$, Corollary \ref{thm-closed-lct} gives Corollary \ref{cor-closed-lct}.

    \begin{Example}\label{ex-xi0}
     It is easy to see that $\xi_0:=(1,0,\ldots,0,\ldots)$ satisfies $\xi_0\in\ell_{\mathcal{I}_+(c_o(\varphi)\varphi)_o}$, which implies that $\xi_0$ is a functional that computes $c_o(\varphi)$ when $c_o(\varphi)>0$ according to the proof of Theorem \ref{thm-closed-jn}.
    \end{Example}
    
    It follows from Corollary \ref{cor-closed-lct} and Example \ref{ex-xi0} that the following corollary holds (similar results can be referred to \cite{GZrestriction}).
    
    \begin{Corollary}\label{Cor-limlogK(t)/t}
        Let $D$ be a bounded pseudoconvex domain in $\mathbb{C}^n$, Let $z\in D$ and $\varphi$ a negative plurisubharmonic function on $D$ such that $\varphi(z)=-\infty$. Denote $K_{\{\varphi<-t\}}(\cdot)$ be the Bergman kernel on the domain $\{\varphi<-t\}\subset D$ for any $t\geq 0$. If $c_z(\varphi)>0$, then
        \[\lim_{t\to +\infty}\frac{\log K_{\{\varphi<-t\}}(z)}{t}=c_z(\varphi).\] 
    \end{Corollary}

    For $\psi\equiv 0$, Theorem \ref{thm-closed-jn} also deduces the following corollary related to Corollary \ref{cor-gln-jn}.
    \begin{Corollary}\label{cor-closed-jn}
        Assume $F$ is a local holomorphic germ near $o$ such that $c_o^F(\varphi)>0$. Then there exists $\xi\in\ell_1$ such that $(\xi\cdot F)(o)\neq 0$ and
     \[c_o^F(\varphi)=\gamma_{\xi}(\varphi).\]
    \end{Corollary}
    
    It should be noted that the proof of Theorem \ref{thm-closed-jn} is independent of the strong openness property. Actually, the closedness property of the $\xi-$cse gives an approach to the strong openness property. Recall the following version of the strong openness property of multiplier ideal sheaves.
    \begin{Theorem}[strong openness property, see \cite{GZSOC}]\label{SOP}
        Let $\varphi$ and $\psi$ be plurisubharmonic functions near $o$, and $F$ holomorphic function near $o$. Assume $c_o^F(\varphi,\psi)\in (0,+\infty)$, then
        \[(F,o)\notin\mathcal{I}(c_o^F(\varphi,\psi)\varphi+\psi)_o.\]
        \end{Theorem}

    We claim that Theorem \ref{thm-closed-jn} can induce Theorem \ref{SOP} (the strong openness property). Let $\varphi$, $\psi$ and domain $D$ be as in Section \ref{Introduction}. Assume there exists some $F$ which is a holomorphic function on $D$ such that $c_o^F(\varphi,\psi)=1$ (without loss of generality). Assume that there exists some $\xi\in\ell_1\setminus \tilde{\ell}_{\mathcal{I}(\psi)_o}$ such that $(\xi\cdot F)(o)\neq 0$ and
     \[c_o^F(\varphi,\psi)=\gamma_{\xi}(\varphi,\psi).\]
    Then Proposition \ref{log-concavity} shows that $-\log K^{\psi}_{\xi}(t)+t$ is (concave and) increasing with respect to $t\in [0,+\infty)$. Combining with Fubini's Theorem, we have (see e.g. \cite{G16, BG1, BG2})
    \begin{equation*}
        \begin{split}
            \int_D|F|^2e^{-\frac{\varphi}{q}-\psi}&=\int_{-\infty}^{+\infty}\left(\int_{\{\frac{\varphi}{q}<-t\}\cap D}|F|^2e^{-\psi}\right)e^t\mathrm{d}t\\
            &\geq \int_0^{+\infty}\frac{|(\xi\cdot F)(o)|^2}{K^{\psi}_{\xi,\varphi}(qt)}e^t\mathrm{d}t+\int_{-\infty}^0\frac{|(\xi\cdot F)(o)|^2}{K^{\psi}_{\xi,\varphi}(0)}e^t\mathrm{d}t\\
            &\geq \frac{|(\xi\cdot F)(o)|^2}{K^{\psi}_{\xi,\varphi}(0)}\left(\int_0^{+\infty}e^{(1-q)t}\mathrm{d}t+\int_{-\infty}^0e^t\mathrm{d}t\right)\\
            &=\frac{q}{q-1}\cdot\frac{|(\xi\cdot F)(o)|^2}{K^{\psi}_{\xi,D}(o)}, \ \forall q>1.
        \end{split}
    \end{equation*}
    Note that $(\xi\cdot F)(o)\neq 0$, and $K^{\psi}_{\xi,D}(o)<+\infty$ (see \cite{BG2}). Then by monotone convergence theorem, we have $\int_D|F|^2e^{-\varphi-\psi}=+\infty$ by letting $q\to 1^+$. This gives Theorem \ref{SOP} (the strong openness property). Similar discussions can be given to Corollary \ref{cor-closed-lct} (related to the openness property, see \cite{Ber13}), Corollary \ref{thm-closed-lct}, and Corollary \ref{cor-closed-jn}.

    \section{Examples of toric plurisubharmonic functions}\label{section_toric}
    In this section, we do some calculations on toric plurisubharmonic functions, which gives an example for our definitions and results of $\xi-$Bergman kernels and the $\xi-$cse. For more details about the discussion in this section, it can be referred to \cite{Guenancia}.
    
    We first show some calculations related to the following shape of toric plurisubharmonic functions (which only depends on $|z_1|,\ldots,|z_n|$):
    \[\varphi_w(z):=\max_{1\leq i\leq n}\left\{\frac{1}{w_i}\log|z_i|\right\},\]
    where $w:=(w_1,\ldots,w_n)\in\mathbb{R}_+^n$, $z=(z_1,\ldots,z_n)\in\Delta^n$, and $\Delta^n$ is the unit polydisc in $\mathbb{C}^n$. Then it can be seen that for any $t\geq 0$,
    \[\{\varphi_w(z)<-t\}=\prod_{i=1}^n\{|z_i|<e^{-w_it}\}.\]
    Then for any $F=\sum_{\alpha}c_{\alpha}z^{\alpha}\in A^2(\{\varphi_w<-t\})$, we have
    \begin{equation*}
    \begin{split}
        \int_{\{\varphi_w<-t\}}|F|^2&=\sum_{\alpha}|c_{\alpha}|^2\int_{\{\varphi_w<-t\}}|z^{\alpha}|^2\\
        &=\sum_{\alpha}|c_{\alpha}|^2\left(\prod_{i=1}^n\int_{\{|z_i|<e^{-w_it}\}}|z_i|^{2\alpha_i}\right)\\
        =&\sum_{\alpha}|c_{\alpha}|^2\left(\pi^n\frac{e^{-2\sum_{i=1}^n(\alpha_i+1)w_it}}{\prod_{i=1}^n(\alpha_i+1)}\right).
        \end{split}
    \end{equation*}
    For any $\alpha\in\mathbb{N}^n$, $t\geq 0$, denote
    \[d_{\alpha}(t):=\int_{\{\varphi_w<-t\}}|z^{\alpha}|^2=\pi^n\frac{e^{-2\sum_{i=1}^n(\alpha_i+1)w_it}}{\prod_{i=1}^n(\alpha_i+1)}.\]
    Then for any $\xi=(\xi_{\alpha})_{\alpha}\in\ell_1$, by Cauchy-Schwarz inequality we have
    \begin{equation*}
        \begin{split}
        |(\xi\cdot F)(o)|^2&=\left|\sum_{\alpha}c_{\alpha}\xi_{\alpha}\right|^2\\
        &=\left|\sum_{\alpha}c_{\alpha}\sqrt{d_{\alpha}(t)}\cdot\frac{\xi_{\alpha}}{\sqrt{d_{\alpha}(t)}}\right|^2\\
        &\leq \left(\sum_{\alpha}|c_{\alpha}|^2d_{\alpha}(t)\right)\cdot\left(\sum_{\alpha}\frac{1}{d_{\alpha}(t)}|\xi_{\alpha}|^2\right)\\
        &\leq\left(\sum_{\alpha}\frac{1}{d_{\alpha}(t)}|\xi_{\alpha}|^2\right)\int_{\{\varphi_w<-t\}}|F|^2,
        \end{split}
    \end{equation*}
    where the equality holds if and only if $F=\lambda\sum_{\alpha}\frac{\overline{\xi_{\alpha}}}{d_{\alpha}(t)}z^{\alpha}\in A^2(\{\varphi_{w}<-t\})$ for some $\lambda\in\mathbb{C}$. Then we get
    \begin{equation*}
        K_{\xi,\varphi_w}(t)=\sum_{\alpha}\frac{1}{d_{\alpha}(t)}|\xi_{\alpha}|^2.
    \end{equation*}
    Note that
    \begin{equation*}
    \begin{split}
            \frac{\log d_{\alpha}(t)}{t}&=-2\sum_{i=1}^n(\alpha_i+1)w_i+\left(\log\frac{\pi^n}{\prod_{i=1}^n(\alpha_i+1)}\right)\frac{1}{t}\\
            &\to -2\sum_{i=1}^n(\alpha_i+1)w_i, \ t\to +\infty.
            \end{split}
    \end{equation*}
    Denote that
    \[\langle \alpha+\mathbbm{1}, w\rangle:=\sum_{i=1}^n(\alpha_i+1)w_i.\]
    Since $\xi\in\ell_1$, we can conclude that
    \begin{equation*}
    \gamma_{\xi}(\varphi_w)=\lim_{t\rightarrow +\infty}\frac{\log K_{\xi,\varphi_w}(t)}{t}=2\max_{\xi_{\alpha}\neq 0}\langle \alpha+\mathbbm{1}, w\rangle.
    \end{equation*}
    
    Recall that (see \cite{BG3})
    \[\ell_0:=\left\{\xi=(\xi_{\alpha})_{\alpha}\in\ell_1 : \exists k\in\mathbb{N}, \ \text{s.t.\ }\xi_{\alpha}=0, \ \forall |\alpha|>k\right\}.\]
    Then it can be seen that $\gamma_{\xi}(\varphi_w)=+\infty \Leftrightarrow \xi\in\ell_1\setminus\ell_0$. In addition, we note that
    \[\min_{\xi\in\ell_1\setminus\{0\}}\gamma_{\xi}(\varphi_w)=2\sum_{i=1}^nw_i,\]
    and
    \[\min_{\xi\in\ell_1,\ (\xi\cdot F)(o)\neq 0}\gamma_{\xi}(\varphi_w)=2\min_{c_{\alpha}\neq 0}\langle \alpha+\mathbbm{1}, w\rangle\]
    for $F=\sum_{\alpha}c_{\alpha}z^{\alpha}\in\mathcal{O}_o$. It is well-known that
    \[c_o(\varphi_w)=2\sum_{i=1}^nw_i,\]
    and
    \[c_o^F(\varphi_w)=2\min_{c_{\alpha}\neq 0}\langle \alpha+\mathbbm{1}, w\rangle\]
    for the above $F$. Now direct calculations show that $\varphi_w$ actually satisfies Corollary \ref{cor-gln-lct} and \ref{cor-closed-jn}. Moreover, the above calculations also show that $\varphi_w$ satisfies Corollary \ref{cor-closed-lct} and \ref{cor-closed-jn}. 
    
    We continue to give some calculations for general toric plurisubharmonic functions. Before the discussions, we consider the following integral:
    \[I^{\textbf{k}}_a(s):=\int_{\{\sum_{i=1}^na_ix_i>s \ \& \ x_i>0, \forall 1\leq i\leq n\}}e^{-\sum_{i=1}^n k_ix_i}\mathrm{d}x_1\cdots\mathrm{d}x_n,\]
    where $\textbf{k}:=(k_1,\ldots,k_n)\in\mathbb{N}^n$, $a:=(a_1,\ldots,a_n)\in\mathbb{R}_+^n$, and $s\in [0,+\infty)$. It is easy to see that for any $i\in\{1,\ldots,n\}$,
    \begin{equation*}
        \begin{split}
        I^{\textbf{k}}_a(s)&\geq \int_{\{a_ix_i>s \ \& \ x_j>0, \forall j\neq i\}}e^{-\sum_{i=1}^n k_ix_i}\mathrm{d}x_1\cdots\mathrm{d}x_n\\
        &=\frac{1}{k_1\cdots k_n}e^{-\frac{k_i}{a_i}s},
        \end{split}
    \end{equation*}
    which implies that
    \begin{equation}\label{I_a^k(s)}
        I^{\textbf{k}}_a(s)\geq \frac{1}{k_1\cdots k_n}e^{-\left(\min_{1\leq i\leq n}\frac{k_i}{a_i}\right)s}.
    \end{equation}

    Now let $\varphi$ be a negative toric plurisubharmonic function (recall that this means $\varphi(z)$ only depends on $|z_1|,\ldots$, and $|z_n|$) on $\Delta^n$, the following lemma shows that $\varphi$ can be written as $\varphi(z)=f(\log|z_1|,\ldots,\log|z_n|)$, where $f$ is convex and non-decreasing in each variable:
    \begin{Lemma}[see \cite{DemaillyAG, Guenancia}]
    Let $\varphi$ be a toric plurisubharmonic function on $\Delta^n$. Then there exists a convex function $f$, non-decreasing in each variable, defined on $(-\infty, 0)$ such that for all $z=(z_1,\ldots, z_n)\in\Delta^n$, we have $\varphi(z)=f(\log|z_1|,\ldots,\log|z_n|)$.
    \end{Lemma}
    Moreover, for the toric plurisubharmonic function $\varphi(z)=f(\log|z_1|,\ldots,\log|z_n|)$ as above, denote $g(\textbf{x}):=-f(-\textbf{x})$. Then $g$ is a function defined on $\mathbb{R}_+$ which is concave, positive, and non-increasing in each variable. And the Newton convex body of $g$ is defined as:
    \[P(g):=\{\lambda\in\mathbb{R}^n : g(\textbf{x})-\langle \lambda, \textbf{x} \rangle \leq O(1) \ \text{as\ } |\textbf{x}|\to +\infty\}.\]
    In addition, denote that $P(\varphi):=P(g)$. We give the following estimate of $\gamma_{\xi}(\varphi)$ for toric plurisubharmonic function $\varphi$:
    \begin{Proposition}\label{estimate-gln-toric}
        Assume $\lambda:=(\lambda_1,\ldots,\lambda_n)\in\mathbb{R}_+^n$ such that $\lambda\in P(\varphi)$, and $\xi\in\ell_1$. Then we have
        \[\gamma_{\xi}(\varphi)\geq 2\max_{\xi_{\alpha}\neq 0}\left(\min_{1\leq i\leq n}\frac{\alpha_i+1}{\lambda_i}\right).\]
    \end{Proposition}
    \begin{proof}
    By definition, it is easy to see that $\gamma_{\xi}(\varphi_1)\leq \gamma_{\xi}(\varphi_2)$ if $\varphi_1\leq \varphi_2+O(1)$ near $o$. Then we have
    \[\gamma_{\xi}(\varphi)\geq \gamma_{\xi}\left(\sum_{i=1}^n\lambda_i\log|z_i|\right),\]
    for $\varphi(z)=-g(-\log|z_1|,\ldots,-\log|z_n|)$ and $\lambda\in P(\varphi)$. By direct calculation and using (\ref{I_a^k(s)}), we get
    \begin{equation}
    \begin{split}
        d_{\alpha}(t):&=\int_{\{\sum_{i=1}^n\lambda_i\log|z_i|<-t\}}|z^{\alpha}|^2\\
        &=(2\pi)^n\int_{\{\sum_{i=1}^n\lambda_ix_i>t \ \& \ x_i>0, \forall 1\leq i\leq n\}}e^{-\sum_{i=1}^n(2\alpha_i+2)x_i}\mathrm{d}x_1\cdots\mathrm{d}x_n\\
        &\geq C_{\lambda}^{\alpha}e^{-\left(\min_{1\leq i\leq n}\frac{2\alpha_i+2}{\lambda_i}\right)t},
        \end{split}
    \end{equation}
    for some positive constant $C_{\lambda}^{\alpha}$. By the similar discussion as above for toric plurisubharmonic function of shape $\varphi_w:=\max_{1\leq i\leq n}\frac{1}{w_i}\log|z_i|$, we are also able to get
    \[\gamma_{\xi}\left(\sum_{i=1}^n\lambda_i\log|z_i|\right)\geq\max_{\xi_{\alpha}\neq 0}\left(-\lim_{t\to+\infty}\frac{d_{\alpha}(t)}{t}\right)\geq 2\max_{\xi_{\alpha}\neq 0}\left(\min_{1\leq i\leq n}\frac{\alpha_i+1}{\lambda_i}\right).\]
    Then it follows that
    \[\gamma_{\xi}(\varphi)\geq 2\max_{\xi_{\alpha}\neq 0}\left(\min_{1\leq i\leq n}\frac{\alpha_i+1}{\lambda_i}\right).\]
    \end{proof}
    
    We can also write Proposition \ref{estimate-gln-toric} as
    \begin{equation}
        \gamma_{\xi}(\varphi)\geq 2\max_{\lambda\in P(\varphi)\cap \mathbb{R}^n_+}\max_{\xi_{\alpha}\neq 0}\left(\min_{1\leq i\leq n}\frac{\alpha_i+1}{\lambda_i}\right).
    \end{equation}
    
    As a corollary of Proposition \ref{estimate-gln-toric}, we give a new proof of the following partly result of Theorem 1.13 in \cite{Guenancia}.
    \begin{Corollary}[see \cite{Guenancia}]\label{int(P(varphi))}
        Let $\varphi$ be a toric plurisubharmonic function on $\Delta^n$. Then we have
    \[z^{\alpha}=z_1^{\alpha_1}\cdots z_n^{\alpha_n}\in\mathcal{I}(2\varphi)_o \Leftarrow \alpha+\mathbbm{1}\in \text{int}\left(P(\varphi)\right),\]
    where $\text{int}\left(P(\varphi)\right)$ denotes the interior of $P(\varphi)$.
    \end{Corollary}
    \begin{proof}
    We can assume that $\varphi$ is negative. By Corollary \ref{cor-gln-jn}, it suffices to prove for any $\xi\in\ell_1$ with $(\xi\cdot z^{\alpha})(o)\neq 0$, we have $\gamma_{\xi}(\varphi)>2+\delta$ for some $\delta>0$. Since $\alpha+\mathbbm{1}\in \text{int}\left(P(\varphi)\right)$, there exists some $\epsilon\in (0,1)$ such that $\alpha+(1-\epsilon)\mathbbm{1}\in P(\varphi)$. Then for any $\xi$ with $(\xi\cdot z^{\alpha})(o)\neq 0$, we have $\xi_{\alpha}\neq 0$, and it follows from Proposition \ref{estimate-gln-toric} that
    \[\gamma_{\xi}(\varphi)\geq 2\max_{\xi_{\beta}\neq 0}\left(\min_{1\leq i\leq n}\frac{\beta_i+1}{\alpha_i+1-\epsilon}\right)\geq 2\min_{1\leq i\leq n}\frac{\alpha_i+1}{\alpha_i+1-\epsilon},\]
    where $\delta:=\min_{1\leq i\leq n}\frac{\alpha_i+1}{\alpha_i+1-\epsilon}-1>0$. The proof is done.
    \end{proof}
    
    Actually, the converse of Corollary \ref{int(P(varphi))} is also true, where more details can be referred to \cite{Guenancia}. 
    
    \section{Valuative point of view}\label{valuative-section}
    We recall the Kiselman numbers (see \cite{DemaillyAG}). Let $\varphi$ be a plurisubharmonic function on $\Delta^n$, and denote $\varphi_w:=\max_{1\leq i\leq n}\{\frac{1}{w_i}\log|z_i|\}$ as above. The Kiselman number of $\varphi$ with respect to $w$ is denoted by
    \[\nu_{w}(\varphi):=\sup\{a\geq 0 : \varphi\leq a\varphi_{w}+O(1) \ \text{near\ } o\}.\]
    The whole set of such $\nu_w$ for all $w\in \mathbb{R}_+^n$ is denoted by $\mathcal{V}_{\mathrm{m}}$, which is also the space of monomial valuations. Actually, it is defined that
    \[\nu_w(f):=\nu_w(\log|f|)=\min_{f^{(\alpha)}(o)\neq 0}\langle\alpha,w\rangle=\min_{f^{(\alpha)}(o)\neq 0}\sum_{i=1}^nw_i\alpha_i.\]
    The thinness of those valuations is denoted by
    \[A(\nu_w):=|w|=\sum_{i=1}^nw_i.\]
    The following well-known characterization of the multiplier ideal is presented in \cite{BFJ08}:
    \begin{Theorem}[see \cite{BFJ08}]\label{thm-valuation-BFJ08}
    Let $f$ be a holomorphic function near $o$ such that $f\in \mathcal{I}(2\varphi)_o$. Then for any $\nu\in \mathcal{V}_{\mathrm{m}}$, we have
    \[\frac{\nu(\varphi)}{\nu(f)+A(\nu)}<1.\]
    \end{Theorem}
    
    Inspired by Theorem \ref{thm-valuation-BFJ08}, we obtain the following estimate for $\xi-$cse. For any $\nu_w\in\mathcal{V}_{\mathrm{m}}$ and any $\xi\in\ell_1$, denote that
    \[\nu_w(\xi):=\max_{\xi_{\alpha}\neq 0}\langle \alpha, w\rangle.\]
    \begin{Proposition}\label{thm-nu(xi)}
    For any $\nu\in\mathcal{V}_{\mathrm{m}}$ and $\xi\in\ell_1$, we have
    \begin{equation}\label{ineq-nu(xi)}
        \frac{\nu(\varphi)}{\nu(\xi)+A(\nu)}\cdot\gamma_{\xi}(2\varphi)\leq 1.
    \end{equation}
    Here $0\cdot\infty:=0$.
    \end{Proposition}
    
    \begin{proof}
    Let $\nu_w\in\mathcal{V}_{\mathrm{m}}$ for any fixed $w\in \mathbb{R}_+^n$. We may assume $\nu_w(\varphi)>0$, otherwise the result is trivial. Then for any $a\in (0,\nu_w(\varphi))$, we have $\varphi\leq a\varphi_w+O(1)$ near $o$. It means that for any $\xi\in\ell_1$,
    \[\gamma_{\xi}(2\varphi)\leq \gamma_{\xi}(2a\varphi_w)=\frac{1}{2a}\gamma_{\xi}(\varphi_w)=\frac{1}{a}\max_{\xi_{\alpha}\neq 0}\langle \alpha+\mathbbm{1},w\rangle.\]
    The last equality comes from the calculations in Section \ref{section_toric}. It follows that
    \[\gamma_{\xi}(2\varphi)\leq\inf_{a\in (0,\nu_w(\varphi))}\frac{1}{a}\max_{\xi_{\alpha}\neq 0}\langle \alpha+\mathbbm{1},w\rangle=\frac{1}{\nu_w(\varphi)}\max_{\xi_{\alpha}\neq 0}\langle \alpha+\mathbbm{1},w\rangle.\]
    In addition, we have
    \[\nu_w(\xi)+A(\nu_w)=\max_{\xi_{\alpha}\neq 0}\langle\alpha+\mathbbm{1},w\rangle.\]
    Then we summarize that
    \begin{equation*}
        \frac{\nu_w(\varphi)}{\nu_w(\xi)+A(\nu_w)}\cdot\gamma_{\xi}(2\varphi)\leq 1.
    \end{equation*}
    \end{proof}
    
    Inequality (\ref{ineq-nu(xi)}) can also be written as
    \[\gamma_{\xi}(\varphi)\leq 2\inf_{\nu\in\mathcal{V}_{\mathrm{m}}}\frac{\nu(\xi)+A(\nu)}{\nu(\varphi)}.\]
    
    Additionally, following from Proposition \ref{thm-nu(xi)}, we can obtain a new proof of the following corollary, which is also a corollary of Theorem \ref{thm-valuation-BFJ08}.
    \begin{Corollary}[see \cite{BFJ08}]\label{cor-nu}
    Let $(f,o)\in\mathcal{O}_o$. Then for any $\nu\in\mathcal{V}_{\mathrm{m}}$, we have
    \[\frac{\nu(\varphi)}{\nu(f)+A(\nu)}\cdot c_o^f(2\varphi)\leq 1.\]
    Here $0\cdot\infty:=0$.
    \end{Corollary}
    \begin{proof}
    Assume $\nu_w(f)=\langle \alpha,w\rangle$, where $\alpha\in\mathbb{N}^n$. Let $\eta=(\eta_{\beta})_{\beta}\in\ell_1$ such that
    \begin{equation*}
        \eta_{\beta}=\left\{
        \begin{array}{ll}
            1, & \beta=\alpha \\
            0, & \beta\neq\alpha
        \end{array}
        \right. ,
    \end{equation*}
    Then it is clear that $(\eta\cdot f)(o)\neq 0$ and $\nu_w(\eta)=\nu_w(f)$. Thus it follows from Theorem \ref{thm-nu(xi)} that
    \begin{equation*}
        \frac{\nu_w(\varphi)}{\nu_w(f)+A(\nu_w)}\cdot\gamma_{\eta}(2\varphi)\leq 1.
    \end{equation*}
    According to Corollary \ref{cor-gln-jn}, we know $c_o^f(2\varphi)\leq \gamma_{\eta}(2\varphi)$. Consequently we obtain
    \begin{equation*}
        \frac{\nu_w(\varphi)}{\nu_w(f)+A(\nu_w)}\cdot c_o^f(2\varphi)\leq 1.
    \end{equation*}
    \end{proof}

    \section{Restriction formula and subadditivity property}\label{restriction-section}
    In this section, we generalize the restriction formula and subadditivity property for complex singular exponents  (see \cite{GZrestriction, DEL, DK}) to $\xi-$cse case.
    
    Let $\Delta^n$ be the unit polydisc in $\mathbb{C}^n$ with coordinate $(z_1,\ldots,z_n)$, and the origin $o=(0,\ldots,0)\in\Delta^n$. Let $H:=\{z_{k+1}=\cdots=z_n=0\}$, where $k$ is an integer such that $1\leq k<n$. Let $o'$ be the origin in $H$. In this section, we denote $(\mathcal{O}_{n,o})^{\text{dual}}$ by $\ell_1^{(n)}$, and $(\mathcal{O}_{k,o'})^{\text{dual}}$ by $\ell_1^{(k)}$. For any $\xi=(\xi_{\alpha})_{\alpha\in\mathbb{N}^k}\in\ell_1^{(k)}$, we denote that $P^*\xi=((P^*\xi)_{\beta})_{\beta\in\mathbb{N}^n}\in\ell_1^{(n)}$ such that
    \begin{equation*}
        (P^*\xi)_{\beta}=\left\{
        \begin{array}{ll}
            \xi_{\alpha}, & \beta=(\alpha_1,\ldots,\alpha_k,0,\ldots,0), \ \text{where\ } \alpha=(\alpha_1,\ldots,\alpha_k) \\
            0, & \exists j\in\{k+1,\ldots,n\}, \ \text{s.t.\ } \beta_{j}\neq 0
        \end{array}
        \right.,
    \end{equation*}
    where $P$ is the projection from $\mathbb{C}^n$ to $H$.
    
    We state the following result restriction formula for $\xi-$cse.
    \begin{Theorem}\label{thm-restriction}
        Let $\varphi$ be a negative plurisubharmonic function on $\Delta^n$ such that $\varphi(o)=-\infty$, $\psi$ any plurisubharmonic function on $\Delta^n$. Let $\xi\in \ell_1^{(k)}$. Then we have
        \[\gamma_{\xi}(\varphi|_H,\psi|_H)\leq\gamma_{P^*\xi}(\varphi,\psi).\]
    \end{Theorem}
    
    We give the following result for preparation, where the partial result was proved in \cite{BG2}.
    \begin{Lemma}[see \cite{BG2}]\label{lem-F_0}
        Let $D$ be a domain in $\mathbb{C}^n$, $\psi$ a plurisubharmonic function on $D$. Let $z\in D$. Then for any $\xi\in \ell_1$, if $K^{\psi}_{\xi,D}(z)>0$, there exists a unique holomorphic function $F_0$ on $D$ such that $(\xi\cdot F_0)(z)=1$ and
        \[K^{\psi}_{\xi,D}(z)=\frac{1}{\int_D |F_0|^2e^{-\psi}}.\]
    \end{Lemma}
    
    \begin{proof}
        The proof of the existence of $F_0$ can be referred to \cite{BG2}. Now we give the proof of the uniqueness. Suppose $F_0$ and $\tilde{F}_0$ are two holomorphic functions on $D$ such that $(\xi\cdot F_0)(z)=1$ and
        \[K^{\psi}_{\xi,D}(z)=\frac{1}{\int_D |F_0|^2e^{-\psi}}=\frac{1}{\int_D |\tilde{F}_0|^2e^{-\psi}}.\]
        Let $F:=\frac{F_0+\tilde{F}_0}{2}$. Then we have $(\xi\cdot F)(z)=(\xi\cdot\tilde{F})(z)=1$, and
        \[\left(K^{\psi}_{\xi,D}(z)\right)^{-1}\leq \int_D |F|^2e^{-\psi}=\frac{\int_D |F_0|^2e^{-\psi}+\int_D |\tilde{F}_0|^2e^{-\psi}}{2}-\int_D \left|\frac{F_0-\tilde{F}_0}{2}\right|^2e^{-\psi},\]
    which implies that
    \[\int_D\left|\frac{F_0-\tilde{F}_0}{2}\right|^2e^{-\psi}=0.\]
    We obtain $F_0=\tilde{F}_0$, which shows the uniqueness of $F_0$.
    \end{proof}
    
    For the proof of Theorem \ref{thm-restriction}, we recall Ohsawa-Takegoshi's $L^2$ extension theorem.
    \begin{Theorem}[see \cite{OT1}]\label{OT-L2ext}
    Let $D$ be a bounded pseudoconvex domain in $\mathbb{C}^n$, and $H:=\{z_{k+1}=\cdots=z_n=0\}$. Let $\psi$ be a plurisubharmonic function on $\Delta^n$. Then for any holomorphic function $f$ on $H\cap D$ satisfying
    \[\int_{H\cap D}|f|^2e^{-\psi|_H}<+\infty,\]
    there exists a holomorphic function $F$ on $D$ satisfying $F|_{H\cap D}=f$, and
    \[\int_{D}|F|^2e^{-\psi}\leq C\int_{H\cap D}|f|^2e^{-\psi|_H},\]
    where $C$ is a constant that only depends on the diameter of $D$.
    \end{Theorem}
    
    The optimal version of Theorem \ref{OT-L2ext} (see \cite{Bl13}) was used in \cite{BG2} to prove Proposition \ref{xi-log-psh}.
    
    \begin{proof}[Proof of Theorem \ref{thm-restriction}]
        According to Lemma \ref{lem-F_0}, for any $t\in [0,+\infty)$, let $f_t$ be the holomorphic function on $\{\varphi|_H<-t\}$ such that $(\xi\cdot f_t)(o')=1$ and
        \[K_{\xi,\varphi|_H}^{\psi|_H}(t)=\frac{1}{\int_{\{\varphi|_H<-t\}}|f_t|^2e^{-\psi|_H}}.\]
        Then by Theorem \ref{OT-L2ext}, there exists $F_t\in\{\varphi<-t\}$ ($D:=\{\varphi<-t\}$ in Theorem \ref{OT-L2ext}, where $\{\varphi<-t\}$ is pseudoconvex) such that $F_t|_{\{\varphi|_H<-t\}}=f_t$, and
        \[\int_{\{\varphi<-t\}}|F_t|^2e^{-\psi}\leq C\int_{\{\varphi|_H<-t\}}|f_t|^2e^{-\psi|_H}.\]
        Note that for $P^*\xi\in\ell^{(n)}_1$, $F_t|_{\{\varphi<-t\}\cap H}=f_t$ implies that $(P^*\xi\cdot F_t)(o)=(\xi\cdot f_t)(o')=1$. Then we have
        \[K_{P^*\xi,\varphi}^{\psi}(t)\geq\frac{|(P^*\xi\cdot F_t)(o)|^2}{\int_{\{\varphi<-t\}}|F_t|^2e^{-\psi}}\geq \frac{1}{C}K_{\xi,\varphi|_H}^{\psi|_H}(t),\]
        where $C$ is independent of $t$, $f_t$, $\varphi$ and $\psi$. It follows that
        \[\gamma_{P^*\xi}(\varphi,\psi)=\lim_{t\to +\infty}\frac{\log K_{P^*\xi,\varphi}^{\psi}(t)}{t}\geq\lim_{t\to +\infty}\frac{\log \left(\frac{1}{C}K_{\xi,\varphi|_H}^{\psi|_H}(t)\right)}{t}=\gamma_{\xi}(\varphi|_H,\psi|_H).\]
    \end{proof}
    
    Let $\xi=(1,0,\ldots,0,\ldots)$ and $\psi\equiv 0$, then by Example \ref{ex-xi0}, Theorem \ref{thm-restriction} degenerates to the restriction formula for complex singularity exponents:
    \begin{Corollary}[see \cite{DEL}]
        $c_{o'}(\varphi|_H)\leq c_o(\varphi)$, where $\varphi|_H\not\equiv -\infty$.
    \end{Corollary}
    
    Let $o_1\in \Delta^{n_1}$, $o_2\in \Delta^{n_2}$, and let $\pi_i: \Delta^{n_1}\times \Delta^{n_2}\to \Delta^{n_i}$ be projections for $i\in\{1,2\}$. We obtain the following subadditivity property for $\xi-$cse.
    
    \begin{Theorem}\label{thm-subadd}
        Let $\xi_1\in\ell^{n_1}$ and $\xi_2\in\ell^{n_2}$. Let $\varphi_1$, $\varphi_2$ be negative plurisubharmonic functions on $\Delta^{n_1}$ and $\Delta^{n_2}$ such that $\varphi_1(o_1)=\varphi_2(o_2)=-\infty$ respectively. Let $\psi_1$, $\psi_2$ be plurisubharmonic functions on $\Delta^{n_1}$ and $\Delta^{n_2}$ respectively. Then
        \begin{equation}\label{ineq-subadd}
            \gamma_{\xi_1\times\xi_2}\big(\max\{\varphi_1\circ \pi_1,\varphi_2\circ \pi_2\}, \psi_1\circ\pi_1+\psi_2\circ \pi_2\big)=\gamma_{\xi_1}(\varphi_1,\psi_1)+\gamma_{\xi_2}(\varphi_2,\psi_2),
        \end{equation}
        where $\xi_1\times\xi_2\in\ell^{(n_1+n_2)}_1$ such that
        \[(\xi_1\times\xi_2)_{(\alpha,\beta)}=(\xi_1)_{\alpha}\cdot(\xi_2)_{\beta}, \] 
        for any $(\alpha,\beta)\in\mathbb{N}^{n_1+n_2}$ with $ \alpha\in\mathbb{N}^{n_1}, \beta\in\mathbb{N}^{n_2}$.
    \end{Theorem}
    
    \begin{proof}
        This theorem is a direct result of the following generalization of the product property of Bergman kernels (e.g. \cite{DNWZ, Oh}):
        \begin{equation*}
            K_{\xi_1\times\xi_2,\max\{\varphi_1\circ \pi_1,\varphi_2\circ \pi_2\}}^{\psi_1\circ\pi_1+\psi_2\circ \pi_2}(t)=K_{\xi_1,\varphi_1}^{\psi_1}(t)\cdot K_{\xi_2,\varphi_2}^{\psi_2}(t)
        \end{equation*}
        holds for any $t\in[0,+\infty)$.

    We firstly prove 
    \begin{equation*}
        K_{\xi_1\times\xi_2,\max\{\varphi_1\circ \pi_1,\varphi_2\circ \pi_2\}}^{\psi_1\circ\pi_1+\psi_2\circ \pi_2}(t)\geq K_{\xi_1,\varphi_1}^{\psi_1}(t)\cdot K_{\xi_2,\varphi_2}^{\psi_2}(t).
    \end{equation*}
    We may assume $K_{\xi_1,\varphi_1}^{\psi_1}(t)>-\infty$ and $K_{\xi_2,\varphi_2}^{\psi_2}(t)>-\infty$. For any $t\in [0,+\infty)$, according to Lemma \ref{lem-F_0}, there exists holomorphic functions $F_{1,t}$ and $F_{2,t}$ on $\{\varphi_1<-t\}$ and $\{\varphi_2<-t\}$ such that $(\xi_1\cdot F_{1,t})(o_1)=(\xi_2\cdot F_{2,t})(o_2)=1$, and
    \[K_{\xi_1,\varphi_1}^{\psi_1}(t)=\frac{1}{\int_{\{\varphi_1<-t\}}|F_{1,t}|^2e^{-\psi_1}}, \ K_{\xi_2,\varphi_2}^{\psi_2}(t)=\frac{1}{\int_{\{\varphi_1<-t\}}|F_{2,t}|^2e^{-\psi_2}}.\]
    Let $F_t:=\pi_1^*(F_{1,t})\cdot\pi_2^*(F_{2,t})$. Note that
    \[\big\{\max\{\varphi_1\circ \pi_1,\varphi_2\circ \pi_2\}<-t\big\}=\{\varphi_1<-t\}\times\{\varphi_2<-t\}.\]
    Then $F_t$ is a holomorphic function on $\big\{\max\{\varphi_1\circ \pi_1,\varphi_2\circ \pi_2\}<-t\big\}$. In addition, we have
    \begin{equation*}
        \begin{split}
            &\int_{\{\varphi_1<-t\}\times\{\varphi_2<-t\}}|F_t|^2e^{-(\psi_1\circ\pi_1+\psi_2\circ \pi_2)}\\
            =&\int_{\{\varphi_1<-t\}}|F_{1,t}|^2e^{-\psi_1}\cdot\int_{\{\varphi_2<-t\}}|F_{2,t}|^2e^{-\psi_2}\\
            =&\left(K_{\xi_1,\varphi_1}^{\psi_1}(t)\cdot K_{\xi_2,\varphi_2}^{\psi_2}(t)\right)^{-1},
        \end{split}
    \end{equation*}
    and it can be checked that $\big((\xi_1\times\xi_2)\cdot F_t\big)(o)=1$ where $o=(o_1,o_2)$ is the origin in $\Delta^{n_1+n_2}$. More precisely, if we write
    \[F_{1,t}(z)=\sum_{\alpha\in\mathbb{N}^{n_1}}a_{\alpha}z^{\alpha}\]
    near $o_1$, and
    \[F_{2,t}(w)=\sum_{\beta\in\mathbb{N}^{n_2}}b_{\beta}w^{\beta}\]
    near $o_2$, then
    \[F_t(z,w)=\sum_{\alpha\in\mathbb{N}^{n_1},\beta\in\mathbb{N}^{n_2}}a_{\alpha}b_{\beta}z^{\alpha}w^{\beta}\]
    near $o$, which implies
    \[\big((\xi_1\times\xi_2)\cdot F_t\big)(o)=\sum_{(\alpha,\beta)\in\mathbb{N}^{n_1}\times\mathbb{N}^{n_2}}(\xi_1)_{\alpha}(\xi_2)_{\beta}a_{\alpha}b_{\beta}=\sum_{\alpha\in\mathbb{N}^{n_1}}(\xi_1)_{\alpha}a_{\alpha}\cdot\sum_{\beta\in\mathbb{N}^{n_2}}(\xi_2)_{\beta}b_{\beta}=1,\]
    where $\xi_1,\xi_2\in\ell_1$ ensures the changes of summation orders. Now we can obtain
    \begin{equation}
        \begin{split}
            K_{\xi_1\times\xi_2,\max\{\varphi_1\circ \pi_1,\varphi_2\circ \pi_2\}}^{\psi_1\circ\pi_1+\psi_2\circ \pi_2}(t)&\geq\frac{|\big((\xi_1\times\xi_2)\cdot F_t\big)(o)|^2}{\int_{\{\max\{\varphi_1\circ \pi_1,\varphi_2\circ \pi_2\}<-t\}}|F_t|^2e^{-(\psi_1\circ\pi_1+\psi_2\circ \pi_2)}}\\
            &=K_{\xi_1,\varphi_1}^{\psi_1}(t)\cdot K_{\xi_2,\varphi_2}^{\psi_2}(t).
        \end{split}
    \end{equation}

    Next we prove
    \begin{equation*}
        K_{\xi_1\times\xi_2,\max\{\varphi_1\circ \pi_1,\varphi_2\circ \pi_2\}}^{\psi_1\circ\pi_1+\psi_2\circ \pi_2}(t)\leq K_{\xi_1,\varphi_1}^{\psi_1}(t)\cdot K_{\xi_2,\varphi_2}^{\psi_2}(t).
    \end{equation*}
    We may assume $K_{\xi_1\times\xi_2,\max\{\varphi_1\circ \pi_1,\varphi_2\circ \pi_2\}}^{\psi_1\circ\pi_1+\psi_2\circ \pi_2}(t)>-\infty$. Similarly, by Lemma \ref{lem-F_0}, there exists unique $F\in A^2(\big\{\max\{\varphi_1\circ \pi_1,\varphi_2\circ \pi_2\}<-t\big\}, e^{-\left(\psi_1\circ\pi_1+\psi_2\circ \pi_2\right)})$ such that $\big((\xi_1\times\xi_2)\cdot F\big)(o)=1$ and
    \[\int_{\big\{\max\{\varphi_1\circ \pi_1,\varphi_2\circ \pi_2\}<-t\big\}}|F|^2e^{-\left(\psi_1\circ\pi_1+\psi_2\circ \pi_2\right)}=\left(K_{\xi_1\times\xi_2,\max\{\varphi_1\circ \pi_1,\varphi_2\circ \pi_2\}}^{\psi_1\circ\pi_1+\psi_2\circ \pi_2}(t)\right)^{-1}.\]
    Then following from Fubini's theorem we have
    \begin{equation*}
        \begin{split}
            &\int_{\big\{\max\{\varphi_1\circ \pi_1,\varphi_2\circ \pi_2\}<-t\big\}}|F|^2e^{-\left(\psi_1\circ\pi_1+\psi_2\circ \pi_2\right)}\\
            =&\int_{\{\varphi_1(z)<-t\}}\int_{\{\varphi_2(w)<-t\}}|F(z,w)|^2e^{-\psi_2(w)}\mathrm{d}\lambda(w)e^{-\psi_1(z)}\mathrm{d}\lambda(z)\\
            \geq& \int_{\{\varphi_1(z)<-t\}}|(\xi_2\cdot F_z)(o_2)|^2\left(K_{\xi_2,\varphi_2}^{\psi_2}(t)\right)^{-1}e^{-\psi_1(z)}\mathrm{d}\lambda(z),
        \end{split}
    \end{equation*}
    where $F_z:=F(z,\cdot)\in A^2(\{\varphi_2<-t\},e^{-\psi_2})$, and $\lambda(z), \lambda(w)$ are Lebesgue measures on $\mathbb{C}^{n_1}, \mathbb{C}^{n_2}$ respectively. In the proof of Theorem \ref{xi-log-psh} in \cite{BG1, BG2}, it is stated that $G(z):=(\xi_2\cdot F_z)(o_2)$ is holomorphic with respect to $z\in\{\varphi_1<-t\}$. It implies that
    \begin{equation*}
           \int_{\{\varphi_1(z)<-t\}}|(\xi_2\cdot F_z)(o_2)|^2e^{-\psi_1(z)}\mathrm{d}\lambda(z)\geq\left(K_{\xi_1,\varphi_1}^{\psi_1}(t)\right)^{-1}|(\xi_1\cdot G)(o_1)|^2.
    \end{equation*}
    Denote $F(z,w)=\sum_{(\alpha,\beta)\in\mathbb{N}^{n_1}\times\mathbb{N}^{n_2}}a_{\alpha,\beta}z^{\alpha}w^{\beta}$ near $o$, then direct calculation shows
    \begin{equation*}
        \begin{split}
            G(z)&=(\xi_2\cdot F_z)(o_2)=\sum_{\beta\in\mathbb{N}^{n_2}}(\xi_2)_{\beta}\sum_{\alpha\in\mathbb{N}^{n_1}}a_{\alpha,\beta}z^{\alpha}\\
            &=\sum_{\alpha\in\mathbb{N}^{n_1}}\left(\sum_{\beta\in\mathbb{N}^{n_2}}(\xi_2)_{\beta}a_{\alpha,\beta}\right)z^{\alpha},
        \end{split}
    \end{equation*}
    and
    \begin{equation*}
        \begin{split}
            (\xi_1\cdot G)(o_1)&=\sum_{\alpha\in\mathbb{N}^{n_1}}\left(\sum_{\beta\in\mathbb{N}^{n_2}}(\xi_2)_{\beta}a_{\alpha,\beta}\right)(\xi_1)_{\alpha}\\
            &=\sum_{(\alpha,\beta)\in\mathbb{N}^{n_1}\times\mathbb{N}^{n_1}}(\xi_1)_{\alpha}(\xi_2)_{\beta}a_{\alpha,\beta}=\big((\xi_1\times\xi_2)\cdot F\big)(o)=1.
        \end{split}
    \end{equation*}
    Then we can conclude that
    \begin{equation*}
        \begin{split}
        &K_{\xi_1\times\xi_2,\max\{\varphi_1\circ \pi_1,\varphi_2\circ \pi_2\}}^{\psi_1\circ\pi_1+\psi_2\circ \pi_2}(t)\\
        =&\left(\int_{\big\{\max\{\varphi_1\circ \pi_1,\varphi_2\circ \pi_2\}<-t\big\}}|F|^2e^{-\left(\psi_1\circ\pi_1+\psi_2\circ \pi_2\right)}\right)^{-1}\\
        \leq& K_{\xi_1,\varphi_1}^{\psi_1}(t)\cdot K_{\xi_2,\varphi_2}^{\psi_2}(t).
        \end{split}
    \end{equation*}

    Now it follows that
    \begin{equation*}
        \begin{split}
        &\gamma_{\xi_1\times\xi_2}(\max\{\varphi_1\circ \pi_1,\varphi_2\circ \pi_2\}, \psi_1\circ\pi_1+\psi_2\circ \pi_2)\\
        =&\lim_{t\to+\infty}\frac{\log K_{\xi_1\times\xi_2,\max\{\varphi_1\circ \pi_1,\varphi_2\circ \pi_2\}}^{\psi_1\circ\pi_1+\psi_2\circ \pi_2}(t)}{t}\\
        =&\lim_{t\to+\infty}\frac{\log K_{\xi_1,\varphi_1}^{\psi_1}(t)}{t}+\lim_{t\to+\infty}\frac{\log K_{\xi_2,\varphi_2}^{\psi_2}(t)}{t}\\
        =&\gamma_{\xi_1}(\varphi_1,\psi_1)+\gamma_{\xi_2}(\varphi_2,\psi_2).
        \end{split}
    \end{equation*}
\end{proof}
    
    Theorem \ref{thm-subadd} can also imply the following subadditivity property for complex singularity exponents with the help of Theorem \ref{thm-restriction} and Example \ref{ex-xi0} (taking $n_1=n_2=n$, $\psi_1=\psi_2=0$, $\xi_1=\xi_2=(1,0,\ldots,0,\ldots)$, and using Theorem \ref{thm-restriction} on the diagonal of $\Delta^n\times\Delta^n$):
    \begin{Corollary}[see \cite{DK}] 
        Let $u, v$ be plurisubharmonic functions on $\Delta^n$, then
        \[c_o(\max\{u,v\})\leq c_o(u)+c_o(v).\]
    \end{Corollary}

\vspace{.1in} {\em Acknowledgements}. We would like to thank Dr. Zhitong Mi for carefully checking this note. We thank the referees for their time and comments. The second named author was supported by National Key R\&D Program of China 2021YFA1003100, NSFC-11825101, NSFC-11522101 and NSFC-11431013. Then third named author was supported by China Postdoctoral Science Foundation BX20230402 and 2023M743719.

\bibliographystyle{alpha}
\bibliography{xbib}

\end{document}